\documentclass[12pt]{amsart}
\usepackage{amssymb}
\input amssym.def
\usepackage{amsmath,amsfonts,hyperref,xcolor,mathtools}
\usepackage{amscd,accents}
\usepackage[mathscr]{eucal}
\usepackage{enumitem}
\usepackage{orcidlink}
\usepackage[utf8]{inputenc}

\hypersetup{colorlinks=true,citecolor={purple},linkcolor={teal},urlcolor={violet}}

\newcommand{\sumdash}{\mathop{{\sum}'}}

\newcommand{\Z}{{\mathbb Z}}

\newcommand{\R}{{\mathbb R}}

\renewcommand{\S}{{ \mathscr S}}
\renewcommand{\a}{{\mathfrak a}}

\newtheorem{thm}{Theorem}[section]
\newtheorem{lem}{Lemma}[section]

\newtheorem{prop}{Proposition}[section]

\newcommand{\thmref}[1]{Theorem~\ref{#1}}
\newcommand{\propref}[1]{Proposition~\ref{#1}}
\newcommand{\lemref}[1]{Lemma~\ref{#1}}

\makeatletter
\@namedef{subjclassname@2020}{%
	\textup{2020} Mathematics Subject Classification}
\makeatother

\begin{document}
	
	\date{\today}
	
	\title[Triple convolution sum of the divisor function]
	{A triple convolution sum of the divisor function}
	
	\author{Bikram Misra \orcidlink{0009-0000-5863-2789}, M. Ram Murty \orcidlink{0000-0003-2086-102X}
		and Biswajyoti Saha \orcidlink{0009-0009-2904-4860}}
	
	\address{Bikram Misra and Biswajyoti Saha\ \newline
		Department of Mathematics, Indian Institute of Technology Delhi, 
		New Delhi 110016, India.}
	\email{bikram.misra@maths.iitd.ac.in, biswajyoti@maths.iitd.ac.in}
	
	\address{M. Ram Murty\ \newline
		Department of Mathematics and Statistics, Queen's University,
		Kingston K7L 3N6, Canada.}
	\email{murty@queensu.ca}
	
	
	
	\begin{abstract}
		We study the triple convolution sum of the 
		divisor function given by
		$$\sum_{n\leq x} d(n)d(n-h)d(n+h)$$
		for $h\neq 0$ and $d(n)$ denotes the 
		number of positive divisors of $n$.
		Based on algebraic and geometric considerations,
		Browning conjectured that the above sum is
		asymptotic to $c_hx(\log x)^3$, for a suitable constant
		$c_h\neq 0$, as $x\to \infty$.  
		This conjecture is still unproved. Using sieve-theoretic results of Wolke and Nair (respectively), it is possible to derive the exact order of the sum. The lower bound of the correct order of magnitude can also be derived by very elementary arguments. In this paper, using the Tauberian theory for multiple Dirichlet series, we prove an explicit lower bound and provide a new theoretical framework to predict Browning's conjectured constant $c_h$.
	\end{abstract}
	
	\maketitle
	
	\section{Introduction}
	The study of convolution sums of arithmetic functions lies at the heart of analytic number theory.
	Ingham studied \cite{AEI} the shifted and additive convolution sums of the divisor function.
	For any integer $n$, let $d(n)$ denote the number of positive divisors of $n$.
	Ingham showed that for a positive integer $h$,
	\begin{equation}\label{shifted-divisor}
		\sum_{n \le N} d(n) d(n+h) = \frac{6}{\pi^2} \sigma_{-1}(h) N (\log N)^2+ O(N \log N)
	\end{equation}
	as $N \to \infty$ and 
	\begin{equation}\label{add-divisor-1}
		\sum_{n < N} d(n) d(N-n) = \frac{6}{\pi^2} \sigma_{1}(N) (\log N)^2 + O(\sigma_1(N) \log N \log\log N)
	\end{equation}
	as $N \to \infty$, where $\sigma_s(n):=\sum_{d \mid n\atop d>0} d^s$ for a complex number $s$.
	
	For a fixed positive integer $h$, the triple convolution sum of the divisor function is defined as
	$$
	S(x;h) := \sum_{n\le x}d(n)d(n-h)d(n+h).
	$$
	It is still an open problem to determine the asymptotic
	behavior of this sum.  
	In \cite{TB}, Browning suggested using algebraic and geometric methods that $S(x;h) \sim c_hx(\log x)^3$ as $x\to\infty$ for a precise constant $c_h>0$, defined as 
	\begin{equation}\label{c_h}
		c_h:=\frac{11}{8}f(h)\prod_{p}\left(1-\frac{1}{p}\right)^2\left(1+\frac{2}{p}\right),
	\end{equation}
	with an explicit function $f$ defined multiplicatively by $f(1)=1$ and
	\begin{equation}\label{f}
		f(p^{\nu}) = \begin{cases}
			\left(1+\frac{4}{p}+\frac{1}{p^2}-\frac{3\nu+4}{p^{\nu+1}}-\frac{4}{p^{\nu+2}}+\frac{3\nu+2}{p^{\nu+3}}\right){\left(1+\frac{2}{p}\right)^{-1}\left(1-\frac{1}{p}\right)^{-2}} & \text{if $p>2,$}\\
			\frac{52}{11}-\frac{41+15\nu}{11\times 2^{\nu}} & \text{if $p=2,$}
		\end{cases} 
	\end{equation}
	for $\nu\ge1$.
	While this conjecture is still open, Browning \cite{TB} proved that for
	$\epsilon>0$ and $H\ge x^{\frac{3}{4}+\epsilon}$, as $x \to \infty$, one has
	\begin{equation}\label{browning on avg}
	\sum_{h\le H}\left(S(x;h) - c_hx(\log x)^3\right) = o(Hx(\log x)^3).
	\end{equation}
	There is a typo  in formula (1.2)
	in the published version of Browning's paper \cite{TB} on page 580.  The factor $(1-{1\over p})^{-1}$ should be $(1-{1\over p})^{-2}$ for $p>2$
	as we have written here in (\ref{f}).  This correction is consistent with our new heuristic derivation below using multiple Dirichlet series.  In fact,  the definition
	in the  \href{https://arxiv.org/pdf/1006.3476}{arXiv version}
    of \cite{TB} agrees with \eqref{f}.
	
	Formula \eqref{browning on avg} suggests that the conjectured asymptotic formula
	for $S(x;h)$ is true ``on average.'' Since divisor functions appear as the Fourier coefficients of Eisenstein series, spectral methods have been extensively used
	by several authors in the study of convolution sums of divisor functions. For example, Blomer \cite{VB} derived a related average result
	with power-saving error terms in this context using spectral tools. A far reaching generalisation of these
	average results for the higher convolutions of a fixed generalised divisor function was recently obtained by
	Matom\"aki, Radziwi\l\l, Shao, Tao and Ter\"av\"ainen \cite{MRSTT}.
	
	\subsection{Statement of the theorems}
    First and foremost, we provide a new framework using the theory of Dirichlet series attached to multivariable arithmetical functions to predict the constant $c_h$ in Browing's conjecture.
	It is worth noting that a lower bound of the correct order of magnitude can be obtained using elementary arguments (which can be seen in the proof of \thmref{main}). However, our result go further to capture an explicit constant in the lower bound. Using the theory of multiple Dirichlet series, we prove the following unconditional lower bound of $S(x;h)$,
	
	\begin{thm}\label{main}
		As $x \to \infty$, we have
		$$
		S(x; h) \geq c_h x(\log x)^3/27 + O_h(x(\log x)^2).
		$$
	\end{thm}
	Our method of dealing with the triple convolution sum $S(x;h)$ leads to the constant $c_h$ as in $\eqref{c_h}$ (see Proposition \ref{second}). This provides a theoretical framework to get to the arithmetic constants appearing in several convolution questions related to divisor function, which is studied in the forthcoming paper \cite{MS}.
	
	The study of the asymptotic behavior of the divisor function at polynomial arguments has classical origin. Using theory of smooth numbers, Erd\H{o}s \cite{erdos} established in 1952 that $\sum_{n\le x}d(F(n))\asymp x\log x$, where $F(t)$ is an irreducible polynomial with integer coefficients. Wolke \cite{DW} generalised this result for a general multiplicative function evaluated at a sequence of natural numbers under certain sieve-theoretic hypotheses. Applying inequality \eqref{inequality-a}, it can be shown from Wolke's results that $S(x;h)\asymp_hx(\log x)^3$. Further generalizations of Wolke's work were  obtained by Nair \cite{Nair}, and subsequently by Nair and Tenenbaum \cite{NT}.
	
	An upper bound of the correct order can also be derived from the main theorem of Nair \cite{Nair}. For the sake of clarity of exposition, we give  a streamlined proof of the upper bound inequality by applying the theory of smooth numbers implicit in the method of
	Erd\H{o}s.  Such a streamlined proof will make the method available for more general problems in number theory.
	
	\begin{thm}\label{main2}
		As $x \to \infty$, we have
		$$
		S(x;h) \ll_h x(\log x)^3.
		$$
	\end{thm}
	
	 
	\section{Preliminaries}
	In this section, we collect the results that are required to prove our theorems. Firstly, we need a
	variant of the Chinese remainder theorem. The key tools for the first theorem come from
	the theory of multiple Dirichlet series and Tauberian theorems due to de la Bret\`eche \cite{RB}. 
	For the second theorem, we need some results on smooth numbers, which is a standard
	chapter in sieve theory.
	
	\subsection{Elementary number theory results}
	The following variant of the Chinese remainder theorem can be found in \cite[Theorem 3-12]{WJL}. 
	Apparently, this non-standard version of the Chinese remainder theorem for
	non-coprime moduli was first written down by the Buddhist monk Yih-hing in 717 CE
	(see pages 57-64 of \cite{LED}). The first formal proof seems to have been written down
	by Stieltjes (of Stieltjes integral fame) as late as 1890. 
	
	\begin{lem}\label{CRT}
		For positive integers $d_1,\ldots,d_k$ and integers $a_1,\ldots,a_k$, the system 
		\begin{equation}\label{lemma-murty-goel}
			\begin{cases}
				x\equiv a_1\bmod{d_1}\\
				\vdots\\
				x\equiv a_k\bmod{d_k}
			\end{cases}\
		\end{equation}
		has a solution if and only if $\gcd (d_i,d_j)\mid(a_i-a_j)$ for all $1\le i,j\le k$.
		Moreover, when a solution exists, it is unique modulo the least common multiple $[d_1,\ldots,d_k]$.
	\end{lem}
	
	Related to the above lemma, 
	we will need various facts regarding the number of solutions 
	of the congruence
	$$f(n)\equiv 0\bmod{a} ,$$
	for a given polynomial $f(x)=a_0x^m + a_1 x^{m-1} + \cdots + a_m \in \Z[x]$.  Such a polynomial is called an integral polynomial.
	It is said to be primitive if the greatest common divisor of
	$a_i$'s equals 1.  
	We quote the following from Nagell \cite[Theorems 52 and 54, page 90]{nagell}.
	\begin{lem} \label{nagell}  Let
		$f(x)$ be a primitive integral polynomial of degree $m$
		with discriminant $D$ different from zero. 
		If $p$ is a prime divisor of $D$, then the
		congruence 
		$$f(x)\equiv 0 \bmod{p^a} $$
		has at most $mD^2$ solutions.  If $p$ is coprime to $D$,
		then the number of solutions is at most $m$.
	\end{lem}
	
	Combining this with the Chinese remainder theorem and specializing
	to the polynomial $f(x)=x(x+h)(x-k)$
	immediately implies the following.
	\begin{lem} \label{b} For $hk(h+k)\neq 0$, 
		the number of solutions of the congruence
		$$ n(n+h)(n-k)\equiv 0 \bmod a$$
		is bounded by $3^{\omega(a)}(hk)^4(h+k)^4,$
		where $\omega(a)$ is the number of distinct
		prime divisors of $a$.
	\end{lem}
	Lemma \ref{nagell} has since been improved by Huxley \cite{huxley} who gave a bound of $m|D|^{1/2}$
	instead of $mD^2$, in the case $p|D$.  Thus, the bound in the above Lemma \ref{b} can be improved to 
	$$3^{\omega(a)} hk(h+k).$$
	We will not need the sharper result but merely record it here
	for academic interest.

	We will need to use the elementary results recorded
	in the next lemma.  They follow from standard analytic number theory
	(see \cite[Ex. 4, Chap. I.3]{tenenbaum}).
	\begin{lem}\label{ant}  Let $\omega(a)$ be the number of distinct
		prime divisors of $a$.  Then as $x \to \infty$, we have
		$$\sum_{a\leq x} 3^{\omega(a)} \ll x(\log x)^2\qquad \text{and} \qquad 
		\sum_{a\leq x} {3^{\omega(a)}\over a} \ll (\log x)^3.$$
	\end{lem}
	
	We will also need the following result from
	elementary number theory.
	\begin{lem} \label{d}  For the divisor function $d(n)$, we have for $hk(h+k)\neq 0$,
		\begin{equation}\label{inequality-a}
		d(n)d(n+h)d(n-k)\leq d(h)d(k) d(h+k)d(n(n+h)(n-k)).
		\end{equation}
	\end{lem}
	
	\begin{proof}
		As usual, we denote by $(u,v)$ and $[u,v]$,
		the gcd and lcm respectively of the integers $u,v$ and
		by $(u,v,w)$ and $[u,v,w]$,
		the gcd and lcm respectively of the integers $u,v,w$.
		Given three non-negative integers $a,b,c$,
		it is evident that
		\begin{equation}\label{max-min}
			\max\{a,b,c\} = a+b+c - \min\{a,b\}-\min\{a,c\} -\min \{b,c\} + \min\{a,b,c\}.
		\end{equation}
		The easiest way to see this is by setting
		$S_r=\{1, \ldots, r\}$ so that the assertion
		is equivalent to
		$$|S_a\cup S_b\cup S_c| = |S_a|+|S_b| +|S_c|
		- |S_a\cap S_b| - |S_a\cap S_c| - |S_b\cap S_c| + |S_a\cap S_b \cap S_c| ,$$
		which follows from the inclusion-exclusion
		principle.  The equality (\ref{max-min})
		implies that for any multiplicative arithmetical
		function $f$, we have 
		\begin{equation}\label{id}
			f(u)f(v)f(w) = {f([u,v,w])f((u,v))f((u,w))f((v,w)) \over f((u,v,w))},
		\end{equation}
		generalizing the familiar identity for two variables,
		$f([u,v])f(u,v)=f(u)f(v),$
		which can be derived using the fact that
		$\max\{a,b\}=a+b-\min\{a,b\}$.
		Applying (\ref{id}) to the divisor function gives
		$$ d(n)d(n+h)d(n-k) = {d([n,n+h,n-k]) d((n,n+h))d((n, n-k))d((n+h,n-k)) \over d((n,n+h,n-k))}.$$
		Thus,
		$$ d(n)d(n+h)d(n-k)\leq d(h)d(k)d(h+k)d(n(n+h)(n-k)),$$
		as claimed.
	\end{proof}
	
	\subsection{Results from multiple Dirichlet series theory}
	
	We now recall the setup of arithmetic functions of several variables and multiple Dirichlet series.
	It seems that arithmetic functions of several variables were first discussed by Vaidyanathaswamy \cite{RV} in 1931.
	He defined a multiplicative function of several variables as a function $f :\mathbb{N}^k\to \mathbb{C}$ that satisfies the property
	$$
	f(m_1n_1,\ldots,m_kn_k) = f(m_1,\ldots,m_k)f(n_1,\ldots,n_k)
	$$ 
	when $\gcd(m_1\cdots m_k, n_1\cdots n_k) = 1$. 
	In the one variable case, we know that Dirichlet convolution of two multiplicative functions is again a multiplicative function.
	The same can be found true if we define the Dirichlet convolution of two arithmetic functions of several variables $f$ and $g$ as
	$$
	(f\star g)(n_1,\ldots,n_k) := \sum_{\substack{d_i|n_i \\ 1\leq i\leq k}} f(d_1,\ldots,d_k)g(n_1/d_1,\ldots,n_k/d_k).
	$$
	In the context of determining the average order of arithmetic functions of several variables,  de la Bret\`eche \cite{RB}
	derived a multi-variable version of the 
	classical Tauberian theorems. To state this
	version, we need the following notations.
	
	\subsection*{Notations}
	Let $\R^+$ denote the set of all non-negative real numbers and $\R_*^+$ denote the set of all positive real numbers.
	For a positive integer $m$, we denote an $m$-tuple $(s_1,\ldots,s_m)$ of complex numbers by $\pmb{s}$. Let $\tau_j=\Im(s_j)$ and
	$\mathcal{L}_m(\mathbb{C})$ be the space of all linear forms on $\mathbb{C}^m$ over $\mathbb{C}$.
	We denote by $\left\{e_j\right\}_{j=1}^{m}$, the canonical basis of $\mathbb{C}^m$ and $\left\{e_j^*\right\}_{j=1}^{m}$,
	the dual basis in $\mathcal{L}_m(\mathbb{C})$. Let $\mathcal{L\mathbb{R}}_m(\mathbb{C})$
	(respectively $\mathcal{L\mathbb{R}}^+_m(\mathbb{C})$) denote the set of linear forms of $\mathcal{L}_m(\mathbb{C})$
	whose restriction to $\mathbb{R}^m$ (respectively to $(\mathbb{R^+})^m$) has values in $\mathbb{R}$
	(respectively, in $\mathbb{R^+}$). Let $\beta_j>0$ for all $j=1,\ldots,m$. Then we denote by $\mathcal{B}$
	the linear form $\sum_{j=1}^{m}\beta_je_j^*$ and $\pmb{\beta}=(\beta_1,\ldots,\beta_m)$ be the associated row matrix.
	We define $X^{\pmb{\beta}} := (X^{\beta_1},\ldots,X^{\beta_m})$. Let $\mathcal{L}$ be a family of linear forms and
	for this we define $\text{conv}(\mathcal{L}) := \sum_{\ell\in \mathcal{L}}\mathbb{R^+}\ell$ and
	$\text{conv}^*(\mathcal{L}) := \sum_{\ell\in \mathcal{L}}\mathbb{(R^*)^+}\ell$.
	With these notations in place, \cite[Th\'eor\`eme 1]{RB} reads as follows:
	
	\begin{thm}\label{theorem 1, breteche}
		Let $f$ be an arithmetic function on $\mathbb{N}^m$ taking positive values and $F$ be the associated Dirichlet series
		$$
		F(\pmb{s}) = \sum_{d_1=1}^{\infty}\ldots\sum_{d_m=1}^{\infty}\frac{f(d_1,\ldots,d_m)}{d_1^{s_1}\ldots d_m^{s_m}}. 
		$$	
		Suppose that there exists $\pmb{\alpha}=(\alpha_1,\ldots, \alpha_m) \in (\mathbb{R^+})^m$ such that $F$ satisfies the following three properties:
		\begin{enumerate}
			\item The series $F(\pmb{s})$ is absolutely convergent for $\pmb{s} \in \mathbb{C}^m$ such that $\Re(s_i)>{\alpha_i}$.
			\item There exists a family $\mathcal{L}$ of $n$ many non-zero linear forms $\mathcal{L}:= \left\{\ell^{(i)}\right\}_{i=1}^{n}$
			in $\mathcal{L\mathbb{R}}^+_m(\mathbb{C})$ and a family of finitely many linear forms $\left\{h^{(k)}\right\}_{k\in \mathcal{K}}$
			in $\mathcal{L\mathbb{R}}^+_m(\mathbb{C})$, such that the function $H$ from $\mathbb{C}^m$ to $\mathbb{C}$ defined by
			$$
			H(\pmb{s}) := F(\pmb{s}+\pmb{\alpha})\prod_{i=1}^{n}\ell^{(i)}(\pmb{s}) 
			$$
			can be extended to a holomorphic function in the domain
			$$
			\mathcal{D}(\delta_1,\delta_3) := \left\{\pmb{s}\in \mathbb{C}^m: \Re\left(\ell^{(i)}(\pmb{s})\right)>-\delta_1
			\ \text{for all} \ i \text{ and } \Re\left(h^{(k)}(\pmb{s})\right)>-\delta_3 \text{ for all } k \in \mathcal K \right\}
			$$
			for some $\delta_1,\delta_3>0$.
			\item There exists $\delta_2>0$ such that for every $\epsilon,\epsilon{'}>0$, the upper bound
			$$
			|H(\pmb{s})| \ll \left(1+||\Im(\pmb{s})||_{1}^{\epsilon}\right)
			\prod_{i=1}^{n} \left(|\Im\left(\ell^{(i)}(\pmb{s})\right)|+1\right)^{1-\delta_2\min\left\{0,\Re\left(\ell^{(i)}(\pmb{s})\right)\right\}}
			$$
			is uniformly valid in the domain $\mathcal{D}\left(\delta_1-\epsilon{'},\delta_3-\epsilon{'}\right)$.  
		\end{enumerate}
		Let $J(\pmb{\alpha}) := \left\{j \in \{1,\ldots,m\} : \alpha_j=0\right\}$.
		Let $r:= |J(\pmb{\alpha})|$ and $\ell^{(n+1)},\ldots,\ell^{(n+r)}$ be the linear forms $e_j^{*}$ for $j \in J(\pmb{\alpha})$.
		Then for ${\pmb \beta}=(\beta_1,\ldots,\beta_m) \in (\mathbb{R^+})^m$,
		there exists a polynomial $Q_{\pmb\beta}\in \mathbb{R}[X]$ of degree less than or equal to
		$n+r-rank\left(\left\{\ell^{(i)}\right\}_{i=1}^{n+r}\right)$ and a real number
		$\theta=\theta\left(\mathcal{L},\left\{h^{(k)}\right\}_{k\in \mathcal{K}},\delta_1,\delta_2,\delta_3,\pmb{\alpha},\pmb{\beta}\right)>0$
		such that we have, for $X\ge1$,	
		$$
		S(X^{\pmb\beta}) := \sum_{1\le d_1\le X^{\beta_1}}\ldots\sum_{1\le d_m\le X^{\beta_m}}f(d_1,\ldots,d_m)
		=X^{\langle\pmb{\alpha},\pmb{\beta}\rangle}\left(Q_{\pmb\beta}(\log X)+O(X^{-\theta})\right).
		$$
	\end{thm}
	
	We also need \cite[Th\'eor\`eme 2]{RB}.
	
	\begin{thm}\label{theorem 2, breteche}
		Let the notations be as in \thmref{theorem 1, breteche}. If we have
		$\mathcal{B}$ is not in the span of $\left\{\ell^{(i)}\right\}_{i=1}^{n+r}$, then $Q_{\pmb\beta}=0$.
		Next suppose, we have the following two conditions:
		\begin{enumerate}
			\item there exists a function $G$ such that $H(\pmb{s}) = G(\ell^{(1)}(\pmb{s}),\ldots,\ell^{(n+r)}(\pmb{s}))$;
			\item $\mathcal{B}$ is in the span of $\left\{\ell^{(i)}\right\}_{i=1}^{n+r}$ and there exists no strict
			subfamily $\mathcal{L'}$ of $\left\{\ell^{(i)}\right\}_{i=1}^{n+r}$ such that $\mathcal{B}$ is in the span of 
			$\mathcal{L'}$ with
			$$
			card(\mathcal{L'})-rank(\mathcal{L'}) = card\left(\left\{\ell^{(i)}\right\}_{i=1}^{n+r}\right)-rank\left(\left\{\ell^{(i)}\right\}_{i=1}^{n+r}\right).
			$$
		\end{enumerate}
		Then, for $X \ge 3$, the polynomial $Q_{\pmb \beta}$ satisfies the relation 
		$$
		Q_{\pmb\beta}(\log X) = C_0X^{-\langle\pmb{\alpha},\pmb{\beta}\rangle}I(X^{\pmb\beta})+O((\log X)^{\rho-1}),
		$$
		where $C_0:= H(0,\ldots,0), \rho:= n+r-rank\left(\left\{\ell^{(i)}\right\}_{i=1}^{n+r}\right)$, and
		$$
		I(X^{\pmb\beta}) := \int\int\ldots\int_{A(X^{\pmb \beta})} \frac{dy_1\ldots dy_n}{\prod_{i=1}^{n} y_i^{1-\ell^{(i)}(\pmb{\alpha})}} 
		$$ 
		with
		$$
		A(X^{\pmb{\beta}}) := \left\{\pmb{y} \in [1,\infty)^n: \prod_{i=1}^{n}y_i^{\ell^{(i)}(e_j)}\le X^{\beta_j}  \ \text{for all} \  j\right\}.
		$$ 
	\end{thm}
	
	\subsection{Results from theory of smooth numbers}	
	We now recall some key results on smooth numbers.
	Let $\S(x,y)$ be the set of natural numbers less than $x$
	that have all their prime factors less than $y$.
	Such numbers are called $y$-smooth.  The size of $\S(x,y)$
	is traditionally denoted by $\Psi(x,y)$.  
	The results we record here can all be found in Chapter III.5 of
	Tenenbaum's book \cite{tenenbaum}.
	The first assertion was first proved by Canfield, 
	Erd\H{o}s and Pomerance in \cite{canfield} (see Corollary
	on page 15).  The second assertion originates in a 1938 paper of Rankin \cite{rankin} and his method of proof is
	often dubbed as ``Rankin's trick.''  The elementary nature
	of this trick is explained well in Granville's survey
	article \cite{granville}.
	\begin{prop}\label{prop-smooth}
		Letting
		$$ u = {\log x \over \log y}, $$
		we have 
		$$\Psi(x,y) \ll x \exp\left(-{1\over 2}u\log u\right),$$
		as $x \to \infty$, provided that for any fixed $\epsilon >0$,
		we have
		$$1 \leq u \leq (1-\epsilon){\log x \over \log\log x}.$$
		Also, for any $A>0$,
		$$\Psi(x, (\log x)^A)\ll x^{1-{1\over A}}\exp\left( {c\log x \over \log \log x}\right),$$  
		as $x \to \infty$, for a positive constant $c$.
	\end{prop}
	
	These results are  very useful.
	In particular, we can apply them to derive the following
	variants of Lemma \ref{ant}.
	
	\begin{lem}\label{smooth-new} 
		For a fixed $\delta \in (0,1)$, $x$ large enough, we have
		$$
		\sum_{\substack{x^{\delta}<d\le x\\ d\in \S(x;(\log x)^2)}}\frac{3^{\omega(d)}}{d} \ll x^{-\frac\delta 2 +\epsilon},
		$$
		for any $\epsilon>0$. Moreover, for a positive integer $t$ such that $x^{1/t}\ge(\log x)^2$, we have
		$$
		\sum_{\substack{x^{\delta}<d\le x\\ d\in \S(x;x^{1/t})}}\frac{3^{\omega(d)}}{d} \ll (\log x)^3\exp(-c t\log t),
		$$
		for some positive constant $c$.
	\end{lem}
	
	\begin{proof} This can be seen as an application of Proposition \ref{prop-smooth} along with partial summation.  
	\end{proof}


	\section{Proof of \thmref{main}}
	
	Before embarking on the proof proper, we exhibit an
	elementary proof of the fact that $S(x;h) \gg_h x (\log x)^3$.  
	We first assume that $h$ is even. Note that if $n$ is coprime
	to $h$, then the numbers $n, n\pm h$ are mutually coprime.
	Therefore, in this case
	$$S(x;h) \geq \sum_{\substack{n\leq x\\ (n,h)=1} }
	d(n)d(n-h)d(n+h) = \sum_{\substack{n\leq x\\ (n,h)=1} }
	d(n(n^2-h^2)).$$
	Any divisor of $n(n^2-h^2)$ factors uniquely as $d_1d_2d_3$ with $d_1|n$, $d_2|(n+h)$ and $d_3|(n-h)$.
	In counting the divisors of $n(n^2-h^2)$,
	we can restrict to counting only divisors $d_1, d_2, d_3$ less
	than $x^{1/3}$ to obtain a lower bound for our
	sum under consideration.  By the coprimality of the terms,
	we must have $n$ belonging to a unique progression 
	(mod $d_1d_2d_3$).  
	Thus, 
	$$ S(x;h) \geq \sumdash_{d_1, d_2, d_3 \leq x^{1/3}} \left(
	{x\over d_1d_2d_3} + O(1) \right), $$
	where the prime on the sum indicates that $d_1, d_2,d_3$
	are coprime to $h$.  This shows that $S(x;h)\gg_h x(\log x)^3.$
	
	Now if $h$ is odd, we consider only even integers $n$ which are coprime
	to $h$. Then the numbers $n, n\pm h$ are mutually coprime.
	Therefore, in this case
	$$S(x;h) \geq \sum_{\substack{\text{even }n\leq x\\ (n,h)=1} }
	d(n)d(n-h)d(n+h) = \sum_{\substack{\text{even }n\leq x\\ (n,h)=1} }
	d(n(n^2-h^2)).$$  
	We can then get a similar estimate as above by choosing the divisor $d_1$ of $n$ to be even.

	To derive the finer result stated in 
	Theorem \ref{main}, we proceed as follows.
	Without loss of generality, we can assume that $h \le x/2$.
	We write the triple convolution sum of the divisor function as follows
	\begin{align*}
		S(x;h) =  \sum_{n\le x}d(n)d(n-h)d(n+h)= \sum_{n\le x}\left( \sum_{u\mid n+h}1\right)\left( \sum_{v\mid n}1\right)\left( \sum_{w\mid n-h}1\right)
		= \sum\limits_{\substack{u\le x+h \\ v\le x \\ w\le x-h}} \sumdash_{n\le x}1,
	\end{align*}
	where the primed sum denotes the sum for $n \le x$ where $n+h\equiv 0 \bmod u, n\equiv 0\bmod v$ and $n-h\equiv 0\bmod w$.
	Note that the outer sum is over all the integer triples $(u,v,w)$ in the set $[1,x+h]\times[1,x]\times[1,x-h].$
	We split the sum to write
	$$
	S(x;h) = S_1(x;h)+S_2(x;h)-S_3(x;h),
	$$
	where
	\begin{align*}
		S_1(x;h) :=  \sum_{\substack{u,v,w\le x }} \  \sumdash_{n\le x}1, \
		S_2(x;h) :=  \sum_{\substack{x<u\le x+h \\ v,w\le x}}\  \sumdash_{n\le x}1 \ \text { and } \
		S_3(x;h) :=   \sum_{\substack{u\le x+h, v\le x \\ x-h<w\le x}}\  \sumdash_{n\le x}1.
	\end{align*}
	It can be seen that $S_2(x;h)=O(\sigma_1(h)\log^2x)$. However, for a lower bound of $S(x;h)$,
	we can just write $S(x;h) \ge S_1(x;h)-S_3(x;h)$.
	Now for $S_3(x;h)$, we note that $1-h\le n-h\le x-h$ and $w\mid (n-h)$.
	If $n-h=mw$ for some $m\ge 1$, then $x-h<w\le mw=n-h\le x-h$, a contradiction.
	Hence, $n-h=mw$ for $m\le 0$. Now suppose $n-h=-tw$ for some $t\ge 1$. Since $x-h<w\le x$, we have
	$-tx\le-tw < -tx+th$ and $1-h\le n-h\le x-h$. Therefore, we must have $-tx+th>1-h$, which implies that
	$h>\frac{1+tx}{1+t}>\frac{x}{2}$, a contradiction. Hence $n=h$ is the only possibility and thus
	$$ S_3(x;h)=
	\sum_{\substack{u\le x+h \\ v\le x \\ x-h<w\le x}}
	\sum_{\substack{2h\equiv0 \bmod u \\ h\equiv 0 \bmod v}}
	1\le h d(h) d(2h).
	$$
	This means that to determine the asymptotic behavior
	of $S(x;h)$ it suffices to study $S_1(x;h)$.
	
	\subsection{Estimating $S_1(x;h)$} 
	Applying \lemref{CRT}, we can write the inner sum
	in the definition of $S_1(x;h)$ as
	$$\sumdash_{n\le x}1 = {xg(u,v,w) \over [u,v,w]} + E(x;u,v,w), $$
	where the error term $E:=E(x;u,v,w)$
	is bounded by 1, the term
	$[u,v,w]$ denotes the least common multiple of $u,v,w$ and $g(u,v,w)$ is a multiplicative function
	taking the value $1$ if and only if the system $n\equiv-h\bmod u,n\equiv0\bmod v,n\equiv h\bmod w$ has a solution,
	else it is $0$.  Therefore,
	$$
	S_1(x;h) =  \sum_{u,v,w\le x}g(u,v,w)\left\{\frac{x}{[u,v,w]}+E\right\}.
	$$

	In order to understand the sum $ \sum_{u,v,w\le x}\frac{g(u,v,w)}{[u,v,w]}$, we analyze the multiple Dirichlet series
	$$
	F(s_1,s_2,s_3):=  \sum\limits_{\substack{ u,v,w\ge 1}}\frac{g(u,v,w)}{[u,v,w]}\frac{1}{u^{s_1}}\frac{1}{v^{s_2}}\frac{1}{w^{s_3}},
	$$ 
	defined for $\Re(s_1),\Re(s_2),\Re(s_3)>1$. As $g$ is a multiplicative function (which can be seen using \lemref{CRT}), the function $F$ has a convergent Euler product
	for $\Re(s_1),\Re(s_2),\Re(s_3)>1$, namely,
	$$
	F(s_1,s_2,s_3) =  \prod_{p} \left( \sum_{\nu_1,\nu_2,\nu_3\ge 0} \frac{g(p^{\nu_1},p^{\nu_2},p^{\nu_3})}{[p^{\nu_1},p^{\nu_2},p^{\nu_3}]}
	\frac{1}{p^{\nu_1s_1+\nu_2s_2+\nu_3s_3}}\right).
	$$
	If at least two of $\nu_1,\nu_2,\nu_3$ are $\ge 1$ and $g(p^{\nu_1},p^{\nu_2},p^{\nu_3})=1$, then $p\mid 2h$.
	So we split the Euler product into two sub-products, one for $p \mid 2h$ and the other for $p\nmid 2h$.
	We first consider the infinite product
	$$
	\prod_{p \ \nmid \ 2h} \left(\sum_{\nu_1,\nu_2,\nu_3\ge 0}\frac{g(p^{\nu_1},p^{\nu_2},p^{\nu_3})}{[p^{\nu_1},p^{\nu_2},p^{\nu_3}]}
	\frac{1}{p^{\nu_1s_1+\nu_2s_2+\nu_3s_3}}\right).
	$$
	Hence in this product if at least two of $\nu_1,\nu_2,\nu_3$ are $\ge 1$, then $g(p^{\nu_1},p^{\nu_2},p^{\nu_3})=0$. Therefore,
	we need to consider the cases when at most one of them is positive and we will get non-zero contributions from
	the triples $(0,0,0), (\nu_1,0,0), (0,\nu_2,0), (0,0,\nu_3)$, where $\nu_i\ge 1$ in the respective cases.
	For $(\nu_1,0,0)$, the contribution is
	$$
	\sum_{\nu_1\ge 1}\frac{1}{[p^{\nu_1},1,1]}\frac{1}{p^{\nu_1s_1}}=\sum_{\nu_1\ge 1}\frac{1}{p^{(1+s_1)\nu_1}} = \frac{1}{p^{1+s_1}-1}.
	$$
	Similarly, for $(0,\nu_2,0)$ and $(0,0,\nu_3)$, the contributions are
	$$ \sum_{\nu_2\ge 1}\frac{1}{p^{(1+s_2)\nu_2}}= \frac{1}{p^{1+s_2}-1} \text{ and }
	\sum_{\nu_3\ge 1}\frac{1}{p^{(1+s_3)\nu_3}}= \frac{1}{p^{1+s_3}-1},$$
	respectively. Therefore, for $p\nmid 2h$ the corresponding Euler product is
	$$
	\prod_{p\ \nmid \ 2h}\left(1+\frac{1}{p^{1+s_1}-1}+\frac{1}{p^{1+s_2}-1}+\frac{1}{p^{1+s_3}-1}\right).
	$$
	We claim that this is same as $\zeta(1+s_1)\zeta(1+s_2)\zeta(1+s_3)$ up to an infinite product which is convergent on the domain
	$$
	\left\{(s_1,s_2,s_3)\in \mathbb{C}^3 : \Re(s_i)>-\frac{1}{2} \  \text{for all} \  i\right\}.
	$$ 
	To see this, we study the product
	$$
	\prod_{p \ \nmid \ 2h}\left(1+\frac{1}{p^{1+s_1}-1}+\frac{1}{p^{1+s_2}-1}+\frac{1}{p^{1+s_3}-1}\right)
	\left(1-\frac{1}{p^{1+s_1}}\right)\left(1-\frac{1}{p^{1+s_2}}\right)\left(1-\frac{1}{p^{1+s_3}}\right).
	$$
	For $X,Y,Z\neq 0,1$, note that
	\begin{align*}
		&\left(1+\frac{1}{X-1}+\frac{1}{Y-1}+\frac{1}{Z-1}\right)\left(1-\frac{1}{X}\right) \left(1-\frac{1}{Y}\right)\left(1-\frac{1}{Z}\right)\\
		&= \left(1-\frac{1}{XY}-\frac{1}{YZ}-\frac{1}{ZX}+\frac{2}{XYZ}\right).
	\end{align*}
	Clearly, the infinite product 
	$$A_h(s_1,s_2,s_3) := \prod_{p \ \nmid \ 2h}\left(1-\frac{1}{p^{s_1+s_2+2}}-\frac{1}{p^{s_2+s_3+2}}-\frac{1}{p^{s_3+s_1+2}}+\frac{2}{p^{s_1+s_2+s_3+3}}\right)
	$$
	converges absolutely if $\Re(1+s_i)>1/2$ for all $i=1,2,3$. We are now ready to apply the Tauberian theorems of
	de la Bret\`eche, \thmref{theorem 1, breteche} and \thmref{theorem 2, breteche}, to derive the following proposition.
	
	\begin{prop}\label{second}
		As $x \to \infty$, we have
		$$
		\sum_{u,v,w\le x}\frac{g(u,v,w)}{[u,v,w]}  = \Delta(h)\prod_{p}\left(1-\frac{1}{p}\right)^2\left(1+\frac{2}{p}\right)(\log x)^3 + O_h((\log x)^2), 
		$$ 
		where $\Delta(h)$ is a non-zero constant given by
		\begin{equation}\label{Delta}
			\Delta(h) = \prod_{p\mid 2h}\left(1-\frac{1}{p}\right)\left(1+\frac{2}{p}\right)^{-1}
			\left(\sum_{\nu_1,\nu_2,\nu_3\ge 0} \frac{g(p^{\nu_1},p^{\nu_2},p^{\nu_3})}{[p^{\nu_1},p^{\nu_2},p^{\nu_3}]}\right).
		\end{equation}
	\end{prop}
	
	Assuming \propref{second}, we first complete our proof.
    Note that
	\begin{align*}
		S_1(x;h) =  \sum_{u,v,w\le x}g(u,v,w)\left\{\frac{x}{[u,v,w]}+E  \right\} \ge \sum_{u,v,w\le x^{1/3}}g(u,v,w)\left\{\frac{x}{[u,v,w]}+O(1)\right\}.
	\end{align*}
	Applying \propref{second}, we therefore have
	$$
	S_1(x;h) \ge \sum_{u,v,w\le x^{1/3}}g(u,v,w)\frac{x}{[u,v,w]}+O(x) \gtrsim  C_hx (\log x)^3/27 +O_h(x(\log x)^2),
	$$
	where 
	$$C_h =\Delta(h)\prod_{p}\left(1-\frac{1}{p}\right)^2\left(1+\frac{2}{p}\right).$$
	So we are left to prove \propref{second} and calculate the constant as claimed in \thmref{main}. Note that,
	for the constant we need to show that $\Delta(h)=11f(h)/8$ (see \eqref{c_h} and \eqref{Delta}).
	
	\subsection{Proof of \propref{second}}
	We need to apply both \thmref{theorem 1, breteche} and \thmref{theorem 2, breteche}. Recall that
	\begin{align*}
		F(s_1,s_2,s_3)&= \prod_{p\nmid 2h}\left( \sum_{\nu_1,\nu_2,\nu_3\ge 0} \frac{g(p^{\nu_1},p^{\nu_2},p^{\nu_3})}{[p^{\nu_1},p^{\nu_2},p^{\nu_3}]}
		\frac{1}{p^{\nu_1s_1+\nu_2s_2+\nu_3s_3}}\right)\times\\
		&\ \ \ \  \prod_{p\mid 2h}\left( \sum_{\nu_1,\nu_2,\nu_3\ge 0} \frac{g(p^{\nu_1},p^{\nu_2},p^{\nu_3})}{[p^{\nu_1},p^{\nu_2},p^{\nu_3}]}\frac{1}{p^{\nu_1s_1+\nu_2s_2+\nu_3s_3}}\right)\\
		&=A_h(s_1,s_2,s_3)B_h(s_1,s_2,s_3)\zeta(1+s_1)\zeta(1+s_2)\zeta(1+s_3)\times\\&\ \ \ \
		\prod_{p\mid 2h}\left( \sum_{\nu_1,\nu_2,\nu_3\ge 0} \frac{g(p^{\nu_1},p^{\nu_2},p^{\nu_3})}{[p^{\nu_1},p^{\nu_2},p^{\nu_3}]}\frac{1}{p^{\nu_1s_1+\nu_2s_2+\nu_3s_3}}\right),
	\end{align*}
	where
	$$
	B_h(s_1,s_2,s_3) := \prod_{p\mid 2h} \left(1-\frac{1}{p^{1+s_1}}\right)\left(1-\frac{1}{p^{1+s_2}}\right)\left(1-\frac{1}{p^{1+s_3}}\right).
	$$
	Note that from the expressions of $A_h(s_1,s_2,s_3)$ and $F(s_1,s_2,s_3)$, we clearly see that $F(s_1,s_2,s_3)$ converges
	if $\Re(s_i) >0$ for all $i=1,2,3$. So we have $\pmb\alpha = (0,0,0)$. Moreover, choosing the family
	$\mathcal{L} =\left\{\ell^{(1)},\ell^{(2)},\ell^{(3)}\right\}$ of non-zero linear forms defined as $\ell^{(i)}(s_1,s_2,s_3) = s_i$ for $i=1,2,3$,
	we get that the function $H(s_1,s_2,s_3)$, defined by
	$$
	H(\textbf{\textit{s}}) :=F(\textbf{\textit{s}}+\pmb{\alpha}) \prod_{i=1}^{3}\ell^{(i)}(\textbf{\textit{s}})=F(s_1,s_ 2,s_3)s_1s_2s_3,
	$$
	can be extended to the domain $\left\{(s_1,s_2,s_3) \in \mathbb{C}^3 : \Re(\ell^{(i)}(s))>-1/2 \text{ for all } i=1,2,3\right\}$. So we
	can choose $\delta_1= \frac{1}{2}$, to apply \thmref{theorem 1, breteche}. Moreover, as the function $H$ is essentially
	the product of the Riemann zeta function and linear forms
	in the $s_i$ variables, the desired growth condition is satisfied using classical estimates (can be found in Titchmarsh
	\cite{titchmarsh}). 
	Since $\pmb\alpha=(0,0,0)$, $r=|J(\pmb\alpha)|=|\left\{j \in \{1,2,3\}:\alpha_j=0\right\}|=3$.
	We consider the linear forms $\{\ell^{(4)},\ell^{(5)},\ell^{(6)}\}$ given by $\ell^{(3+i)}(s_1,s_2,s_3) = e_i^*(s_1,s_2,s_3)=s_i$ for $i=1,2,3$.
	Therefore, using \thmref{theorem 1, breteche}, we conclude that
	\begin{align*}
		\sum_{u\le x^{\beta_1}} \ \sum_{v\le x^{\beta_2}} \ \sum_{w\le x^{\beta_3}} \ \frac{g(u,v,w)}{[u,v,w]}
		&\sim x^{\langle\pmb\alpha,\pmb\beta\rangle}Q_{\pmb\beta}(\log x).
	\end{align*}
	Now we will apply \thmref{theorem 2, breteche} for ${\pmb\beta}=(\beta_1,\beta_2,\beta_3)=(1,1,1)$. The concerned linear form
	is $\mathcal{B}(s_1,s_2,s_3) := \left(\sum_{j=1}^{3}\beta_j e_j^*\right)(s_1,s_2,s_3) = s_1+s_2+s_3$. The polynomial
	$Q_{\pmb\beta}$ satisfies the relation
	$$
	Q_{\pmb\beta}(\log x) = C_0x^{-\langle\pmb\alpha,\pmb\beta\rangle}I(x^{\pmb\beta}) +O((\log x)^{\rho-1}),
	$$
	where $C_0:= H(0,0,0), \rho:=n+r-\text{rank}\left(\left\{\ell^{(i)}\right\}_{i=1}^6\right) = 3$ and
	$$
	I(x^{\pmb\beta}) := \int\int\int_{A(x^{\pmb\beta})} \frac{dy_1dy_2dy_3}{\prod_{i=1}^{3} y_i^{1-\ell^{(i)}(\pmb\alpha)}},
	$$ 
	with 
	$$
	A(x^{\pmb\beta}) := \left\{\pmb{y} \in [1,\infty)^3: \prod_{i=1}^{3}y_i^{\ell^{(i)}(e_j)} \le x^{\beta_j} \  \ \text{for all} \ j\right\}
	= \left\{(y_1,y_2,y_3) : 1\le y_i\le x \  \ \text{for all} \ i \right\}.
	$$
	Therefore,
	$$
	I(x^{\pmb\beta}) = \int_{y_1=1}^{x}\int_{y_2=1}^{x}\int_{y_3=1}^{x} \frac{dy_1dy_2dy_3}{y_1y_2y_3} = (\log x)^3.
	$$
	Further,
	\begin{align*}
		H(0,0,0) &= \prod_{p\nmid 2h}\left(1-\frac{3}{p^2}+\frac{2}{p^3}\right)  \prod_{p\mid 2h}\left(1-\frac{1}{p}\right)^3
		\prod_{p\mid 2h}\left(  \sum_{\nu_1,\nu_2,\nu_3\ge 0} \frac{g(p^{\nu_1},p^{\nu_2},p^{\nu_3})}{[p^{\nu_1},p^{\nu_2},p^{\nu_3}]}\right)\\
		&=\prod_{p\nmid 2h}\left(1-\frac{1}{p}\right)^2\left(1+\frac{2}{p}\right)  \prod_{p\mid 2h}\left(1-\frac{1}{p}\right)^3 
		\prod_{p\mid 2h}\left(  \sum_{\nu_1,\nu_2,\nu_3\ge 0} \frac{g(p^{\nu_1},p^{\nu_2},p^{\nu_3})}{[p^{\nu_1},p^{\nu_2},p^{\nu_3}]}\right)\\
		&=\prod_{p}\left(1-\frac{1}{p}\right)^2\left(1+\frac{2}{p}\right)
		\prod_{p\mid 2h}\left(1-\frac{1}{p}\right)\left(1+\frac{2}{p}\right)^{-1} \left(\sum_{\nu_1,\nu_2,\nu_3\ge 0} \frac{g(p^{\nu_1},p^{\nu_2},p^{\nu_3})}{[p^{\nu_1},p^{\nu_2},p^{\nu_3}]}\right)\\
		&=\Delta(h)\prod_{p}\left(1-\frac{1}{p}\right)^2\left(1+\frac{2}{p}\right).
	\end{align*}
	
	\subsection{Verification of $\Delta(h)=11f(h)/8$}
	We begin with an odd prime factor $p$ of $h$. 
	We have to evaluate the sum
	$$
	s_p:=\sum_{\nu_1,\nu_2,\nu_3\ge 0} \frac{g(p^{\nu_1},p^{\nu_2},p^{\nu_3})}{[p^{\nu_1},p^{\nu_2},p^{\nu_3}]}.
	$$
	Recall that $g(p^{\nu_1},p^{\nu_2},p^{\nu_3}) = 1$ if and only if 
	$(p^{\nu_1}, p^{\nu_2}) \mid h$, $(p^{\nu_2}, p^{\nu_3}) \mid h$ and $   (p^{\nu_1}, p^{\nu_3}) \mid 2h$.
	Let $r$ be the highest power of $p$ dividing $h$. Clearly, if $g(p^{\nu_1},p^{\nu_2},p^{\nu_3})=1$, then at most one of the
	$\nu_i$ can be $\ge r+1$. The contribution in each case to the sum $s_p$ is
	$$
	(r+1)^2\sum_{a \ge r+1}\frac{1}{p^{a}}.
	$$
	Hence the total contribution here would be $3(r+1)^2\sum_{a \ge r+1}\frac{1}{p^{a}}$. Now we can assume $0 \le \nu_i \le r$ for all $i=1,2,3$.
	If all these indices are equal, we get the contribution 
	$$
	\sum_{a=0}^r\frac{1}{p^{a}}.
	$$
	If exactly two of the indices are equal, then the remaining index has $3$ possibilities and that index can be bigger or smaller than the
	other two indices. In any case, we get the contribution
	$$
	6\sum_{a=1}^r \sum_{b=0}^{a-1} \frac{1}{p^{a}}=6\sum_{a=1}^r \frac{a}{p^{a}}.
	$$
	Now if all three indices are different, we can have $6$ possible orderings among them and in this case, we get the contribution
	$$
	6\sum_{a=2}^r \sum_{b=1}^{a-1}\sum_{c=0}^{b-1} \frac{1}{p^{a}}=6\sum_{a=2}^r \sum_{b=1}^{a-1}  \frac{b}{p^{a}}= 3 \sum_{a=2}^r  \frac{a(a-1)}{p^{a}},
	$$
	and hence
	$$
	s_p=3(r+1)^2\sum_{a \ge r+1}\frac{1}{p^{a}}+\sum_{a=0}^r\frac{1}{p^{a}}+6\sum_{a=1}^r \frac{a}{p^{a}}+3 \sum_{a=2}^r  \frac{a(a-1)}{p^{a}}.
	$$
	It is easy to check that
	$$
	(r+1)^2\sum_{a \ge r+1}\frac{1}{p^{a}}= \frac{(r+1)^2}{p^{r+1}} \left(1-\frac{1}{p}\right)^{-1} \text{ and } \sum_{a=0}^r\frac{1}{p^{a}}=\left(1-\frac{1}{p^{r+1}}\right)\left(1-\frac{1}{p}\right)^{-1}.
	$$
	One can also check that
	$$
	\sum_{a=1}^r \frac{a}{p^{a}}=\left(\frac{1}{p} - \frac{r+1}{p^{r+1}} + \frac{r}{p^{r+2}} \right)\left(1-\frac{1}{p}\right)^{-2},
	$$
	and
	$$
	\sum_{a=2}^r  \frac{a(a-1)}{p^{a}}=\left(\frac{2}{p^2} - \frac{r^2+r}{p^{r+1}} + \frac{2r^2-2}{p^{r+2}} - \frac{r^2-r}{p^{r+3}} \right)\left(1-\frac{1}{p}\right)^{-3}.
	$$
	So,
	\begin{align*}
		\left(1-\frac{1}{p}\right)^3 s_p=&\left(\frac{3r^2+6r+2}{p^{r+1}}+1\right)\left(1-\frac{2}{p}+\frac{1}{p^2}\right)+6\left(\frac{1}{p} - \frac{r+1}{p^{r+1}} + \frac{r}{p^{r+2}} \right)\left(1-\frac{1}{p}\right)\\
		&+3\left(\frac{2}{p^2} - \frac{r^2+r}{p^{r+1}} + \frac{2r^2-2}{p^{r+2}} - \frac{r^2-r}{p^{r+3}} \right)\\
		=& 1 + \frac{4}{p} + \frac{1}{p^2} - \frac{3r+4}{p^{r+1}} - \frac{4}{p^{r+2}} + \frac{3r+2}{p^{r+3}}.
	\end{align*}
	This completes the verification of the factors corresponding to the odd prime divisors of $h$, using \eqref{f}.
	For the prime $p=2$, we show that if $r$ is the highest power of $2$ dividing $h$, then the corresponding
	factor in $\Delta(h)$ is $\frac{13}{2}-\frac{41+15r}{2^{r+3}}=\frac{11}{8} \left( \frac{52}{11}-\frac{41+15r}{2^r\times 11} \right)$.
	This will complete the proof.
	
	We evaluate the sum
	$$
	s_2:=\sum_{\nu_1,\nu_2,\nu_3\ge 0} \frac{g(2^{\nu_1},2^{\nu_2},2^{\nu_3})}{[2^{\nu_1},2^{\nu_2},2^{\nu_3}]}.
	$$
	If $0 \le \nu_i \le r$ for all $i=1,2,3$, then as before the contribution to $s_2$ in this case is
	\begin{align*}
		&\sum_{a=0}^r\frac{1}{2^{a}}+6\sum_{a=1}^r \frac{a}{2^{a}}+3 \sum_{a=2}^r  \frac{a(a-1)}{2^{a}}\\
		&=2\left(1-\frac{1}{2^{r+1}}\right) + 24\left(\frac{1}{2} - \frac{r+1}{2^{r+1}} + \frac{r}{2^{r+2}} \right)
		+ 24\left(\frac{1}{2} - \frac{r^2+r}{2^{r+1}} + \frac{2r^2-2}{2^{r+2}} - \frac{r^2-r}{2^{r+3}} \right)\\
		&= 26 - \frac{25}{2^r} - \frac{15r}{2^r} - \frac{3r^2}{2^r}.
	\end{align*}
	Now we assume that $\nu_i \ge r+1$ for at least one of $i=1,2,3$. Here the counting is different.
	We note that if $g(2^{\nu_1},2^{\nu_2},2^{\nu_3})=1$, then we have the following conditions: 
	at most one of $\nu_1,\nu_2 >r$, at most one of $\nu_2,\nu_3 >r$ and at most one of $\nu_1,\nu_3 >r+1$.
	We then have the following four possible choices:
	\begin{enumerate}
		\item[i)] if $\nu_1=r+1$, then $\nu_2 <r+1$ and $\nu_3 \ge 0$;
		\item[ii)] if $\nu_1>r+1$, then $\nu_2 <r+1$ and $\nu_3 \le r+1$;
		\item[iii)] if $\nu_1<r+1$ and $\nu_2 \ge r+1$, then $\nu_3 < r+1$;
		\item[iv)] if  $\nu_1<r+1$ and $\nu_2 < r+1$, then $\nu_3 \ge r+1$.
	\end{enumerate}
	The contribution of case i) to $s_2$ is
	$$
	\frac{(r+1)^2}{2^{r+1}} + (r+1)\sum_{\nu_3 \ge r+1}\frac{1}{2^{\nu_3}}=\frac{(r+1)^2}{2^{r+1}}+\frac{(r+1)}{2^{r}}.
	$$
	The contribution of case ii) to $s_2$ is
	$$
	(r+1)(r+2)\sum_{\nu_1 > r+1}\frac{1}{2^{\nu_1}}=\frac{(r+1)(r+2)}{2^{r+1}}=\frac{(r+1)^2}{2^{r+1}}+\frac{(r+1)}{2^{r+1}}.
	$$
	The contribution of case iii) to $s_2$ is
	$$
	(r+1)^2\sum_{\nu_2 \ge r+1}\frac{1}{2^{\nu_2}}=\frac{(r+1)^2}{2^{r}}.
	$$
	The contribution of case iv) to $s_2$ is
	$$
	(r+1)^2\sum_{\nu_3 \ge r+1}\frac{1}{2^{\nu_3}}=\frac{(r+1)^2}{2^{r}}.
	$$
	Hence the total contribution of these above four cases is $\frac{3(r+1)^2}{2^{r}}+\frac{3(r+1)}{2^{r+1}}$ and thus
	$$
	s_2= 26 - \frac{41+15r}{2^{r+1}}=4\left(\frac{13}{2}-\frac{41+15r}{2^{r+3}}\right),
	$$
	and this completes the proof of \thmref{main}.
	
	\section{Proof of \thmref{main2}}
	In 1952, Erd\H{o}s proved that if $f(x)$ is any
	irreducible polynomial with integer coefficients,
	then there are positive constants $c_1$ and $c_2$ such that
	$$ c_1 x\log x \leq \sum_{n\leq x} d(f(n)) \leq c_2 x\log x,$$
	for $x\geq 2$.  The result is false if
	$f(x)$ is not irreducible as can be 
	seen from Ingham's theorem (\ref{shifted-divisor}).  However, the argument of Erd\H{o}s can be suitably adapted to show that 
	in our case,
	$$S(x;h) \ll_h x(\log x)^3.  $$
	Though it is difficult to identify them in Erd\H{o}s's paper \cite{erdos}, his method has three steps
	and it proceeds as follows.  We consider $f(x)=x(x-h)(x+h)$.
	
	By Lemma \ref{d}, we first note that 
	$$ S(x;h)\ll_h
	\sum_{n\leq x} d(f(n)).$$
	Hence our goal is to bound the latter sum
	applying the method of Erd\H{o}s.
	Suppose that 
	$f(n)$ has $r$ prime factors, counted with multiplicity.  We write 
	\begin{equation}\label{factorization}
		|f(n)|= (p_1 p_2 \cdots p_j)(p_{j+1} \cdots p_r),
	\end{equation}
	with $p_1\leq p_2\leq \cdots \le p_j\leq p_{j+1}\leq \cdots \leq p_r$ and
	the index $j$ is the largest such that 
	$$m:= p_1 p_2 \cdots p_j \leq x.  $$
	The strategy is that the bulk of the divisors
	of $f(n)$ are actually coming from the factor $m=(p_1p_2\cdots p_j)$.  Indeed, 
	by the sub-multiplicativity of the divisor function,
	we have 
	$$d(f(n))\leq 2^{r-j} d(p_1\cdots p_j).$$
	\begin{enumerate}
		\item  {\bf Step 1:}  
		If $n$ is such that $r-j$ is bounded
		by 11 (say), then
		the contribution to the sum $S(x;h)$ from
		such $n$ is bounded by
		$$ \ll \sum_{m\leq x} \sum_{a|f(m), \hspace{0.5mm} a \le x} 1 = \sum_{a\leq x} \, \sum_{a|f(m), \hspace{0.5mm} m\leq x} 1.$$
		By the Chinese remainder theorem as in Lemma \ref{b}, the inner sum is
		easily seen to be
		$$\ll {3^{\omega(a)}x\over a}.$$
		Summing this over $a\leq x$ 
		and using Lemma \ref{ant}, we get the desired bound
		of $O(x\log^3 x)$ coming from these $n$'s.
		\item  {\bf Step 2:}  Suppose now that $n$ is
		such that for $f(n)$, 
		$r-j$ is greater than 11.  
		Choosing $x$ large enough, we get that 
		$p_{j+1}<x^{4/11}$, for
		otherwise $|f(n)|>x^4,$ a contradiction
		since in our case $f$ has degree 3 and $f(n)=O(n^3)$.
		By the definition of $j$, we have 
		$$(p_1\cdots p_j )p_{j+1} >x$$
		so that 
		\begin{equation}\label{bounds}
			x^{7/11} < p_1 \cdots p_j < x.
		\end{equation}
		Now, $p_j \leq p_{j+1} <x^{4/11}$ and so, there is
		a positive integer $t$ such that 
		$$ x^{1/(t+1)} \leq p_j \leq x^{1/t}.$$
		{\bf Step 2a.} We first suppose that $x^{1/t}>(\log x)^2$. Note that $p_1\cdots p_j$ is $x^{1/t}$-smooth. 
		The plan  is to estimate the contribution from each
		possible value of $t$ and then sum over the $t$.  
		Thus, fixing $t$, we see that 
		$$ p_{j+1} \cdots p_r > x^{\frac{r-j}{t+1}}.$$
		For our cubic polynomial, we must have $r-j\leq 3(t+1)$. Thus, for such $n$
		under consideration, we use
        the elementary inequality
\begin{equation}\label{dual}
d(n)\leq 2\sum_{\substack{a|n\\a\ge\sqrt{n}}} 1,
\end{equation}
        to get
		$$d(f(n))< 2^{3t+3} d(p_1\cdots p_j)\leq 2^{3t+4} \sum_{\substack{d|p_1\cdots p_j\\ x\geq d\geq\sqrt{p_1\cdots p_j}} } 1,$$
		with $p_1\cdots p_j$
		being $x^{1/t}$-smooth and the divisor $d > x^{7/22}$ by virtue
		of (\ref{bounds}).  
        We have the total contribution from such numbers is  by Lemma \ref{b}, bounded above by
		$$ \sum_{t \ge 1} 2^{3t+4} \sum_{\substack{d\in \S(x, x^{1/t})\\ d > x^{7/22}}} \sum_{\substack{m\leq x\\ d|f(m)}} 1 \ll \sum_{t \ge 1} 2^{3t} \sum_{\substack{d\in \S(x, x^{1/t})\\ d > x^{7/22}}} {x3^{\omega(d)}\over d}.$$
		The innermost sum is by Lemma \ref{smooth-new},
		$$ \ll x(\log x)^3 \exp\left( - ct\log t\right),$$
		for  some suitable positive constant $c$.  Since
		$$\sum_{t \ge 1} 2^{3t} \exp\left( - ct\log t\right)$$
		converges, we are done in this case.
		\newline
		{\bf Step 2b.}  It remains to consider the case
		$x^{1/t}\leq (\log x)^2$.  
		Using the familiar estimate $d(n) \ll n^\epsilon$, for any
		$\epsilon>0$, we first write
		$$
		d(f(n))
		\ll n^{\epsilon} d(p_1 \cdots p_j) 
		\ll x^\epsilon \sum_{\substack{d|p_1\cdots p_j\\ x\geq d\geq\sqrt{p_1\cdots p_j}} } 1,
		$$
		with $p_1\cdots p_j$
		being $(\log x)^2$-smooth and the divisor $d > x^{7/22}$.
		Now summing for all these $n$'s 
		(indicated by a prime for the $n$'s under consideration), we
		get
		$$\sumdash_{n\leq x} d(f(n)) \ll x^\epsilon
		\sum_{\substack{d\in \S(x, (\log x)^2)\\ d>x^{7/22}}} {x3^{\omega(d)}\over d} \ll x^{1+2\epsilon- 7/44},$$
		by Lemmas \ref{b} and \ref{smooth-new}. This is negligible for $0<\epsilon<7/88$ and hence completes the proof. \qed
	\end{enumerate}

    We remark that in the above computations, the implicit constants can be made explicit.

	\section{A more general triple convolution sum}
	For fixed positive integers $h,k$, we can also consider the triple convolution sum
	$$
	T(x;h,k) := \sum_{n\le x}d(n)d(n+h)d(n-k).
	$$
	As before we can write
	\begin{align*}
		T(x;h,k) =  \sum_{n\le x}d(n)d(n+h)d(n-k)= \sum_{n\le x}\left( \sum_{u\mid n+h}1\right)\left( \sum_{v\mid n}1\right)\left( \sum_{w\mid n-k}1\right)
		= \sum\limits_{\substack{u\le x+h \\ v\le x \\ w\le x-k}} \sumdash_{n\le x} 1,
	\end{align*}
	where the primed sum denotes the sum for $n \le x$ where $n+h\equiv 0 \bmod u, n\equiv 0\bmod v$ and $n-k\equiv 0\bmod w$.
	Note that the outer sum is over all the integer triples $(u,v,w)$ in the set $[1,x+h]\times[1,x]\times[1,x-k].$
	Again we split the sum to write
	$
	T(x;h,k) = T_1(x;h,k)+T_2(x;h,k)-T_3(x;h,k),
	$
	where
	\begin{align*}
		T_1(x;h,k) :=  \sum_{\substack{u,v,w\le x }} \  \sumdash_{n\le x}1, \
		T_2(x;h,k) :=  \sum_{\substack{x<u\le x+h \\ v,w\le x}}\  \sumdash_{n\le x} 1 \ \text { and } \
		T_3(x;h,k) :=   \sum_{\substack{u\le x+h, v\le x \\ x-k<w\le x}}\  \sumdash_{n\le x} 1.
	\end{align*}
	As earlier, the main contribution comes from $T_1(x;h,k)$, which can be studied using
	the multiple Dirichlet series
	$$
	\sum_{ u,v,w\ge 1}\frac{g'(u,v,w)}{[u,v,w]}\frac{1}{u^{s_1}}\frac{1}{v^{s_2}}\frac{1}{w^{s_3}},
	$$ 
	where $g'(u,v,w)$ is a multiplicative function which takes the value $1$ if and only if the system
	$n\equiv-h\bmod u, n\equiv0\bmod v,n\equiv k\bmod w$ has a solution, else it is $0$.
	Again we write this series as a product and split that product into two parts,
	one for the primes dividing the least common multiple $[h,k,(h+k)]$ and the other one
	for the primes not dividing $[h,k,(h+k)]$. Applying \thmref{theorem 1, breteche} and \thmref{theorem 2, breteche},
	we can then derive the following proposition.
	
	\begin{prop}
		As $x \to \infty$, we have
		$$\sum_{u,v,w\le x}\frac{g'(u,v,w)}{[u,v,w]} 
		\sim \Delta(h,k)\prod_{p}\left(1-\frac{1}{p}\right)^2\left(1+\frac{2}{p}\right)(\log x)^3 , 
		$$ 
		where $\Delta(h,k)$ is a non-zero constant given by
		\begin{equation}\label{Delta-hk}
			\Delta(h,k) =\prod_{p\mid [h,k,(h+k)]}\left(1-\frac{1}{p}\right)\left(1+\frac{2}{p}\right)^{-1}
			\left(\sum_{\nu_1,\nu_2,\nu_3\ge 0} \frac{g'(p^{\nu_1},p^{\nu_2},p^{\nu_3})}{[p^{\nu_1},p^{\nu_2},p^{\nu_3}]}\right).
		\end{equation}
	\end{prop}
	As before, the lower bound and upper bound of $T(x;h,k)$ follows from Wolkes's theorems \cite{DW}. 
	We can also derive an explicit lower bound for $T(x;h,k)$ along the lines of the proof of \thmref{main}.
	
	\section{Concluding remarks}
	
	
	Our methods have further applications.  For example,
	we can investigate triple and higher convolution sums
	of the generalised divisor functions $d_k(n)$, which is the number of ways of writing
	$n$ as a product of $k$ factors ($k \ge 2$). This is to appear in the forthcoming paper\cite{MS}. 
	A further variation arises
	if instead of summing over $n\leq x$, we sum over
	primes $p\leq x$. A more involved sum
	$$\sum_{n\leq x} d(f(n)),$$
	where $f(t)$ is a polynomial with
	integer coefficients having a specified number of irreducible factors can also be dealt using these methods.
    This is not totally straightforward and is to be discussed in the doctoral thesis by the first author.

	
	\section*{Acknowledgements}
	The second author thanks IIT Delhi for its warm
	hospitality and pleasant ambiance during his visit when this research
	was initiated.  We thank Tim Browning and 
	Maksym {Radziwi\l\l} 
	for bringing the papers \cite{NT} and \cite{MRSTT} (respectively) to our attention. 	The research of the second author was partially supported by an NSERC Discovery Grant, numbered RGPIN-2020-03927.  
	The research of the last author was partially supported by IIT Delhi IRD project MI02455
	and by SERB (ANRF) grant SRG/2022/001011.

\end{document}